\documentclass[a4paper,11pt, leqno]{article}
\usepackage[latin1]{inputenc}
\usepackage[hidelinks]{hyperref}
\usepackage{articlemacro}

\usepackage[style=alphabetic]{biblatex}
\bibliography{references}
\DeclareFieldFormat{postnote}{#1}
\DeclareFieldFormat{multipostnote}{#1}
\usepackage[total={17cm,24cm}]{geometry}

\begin{document}

\title{Abstract Zip Data}
\author{Christopher Lang}

\hypersetup{
    pdftitle={Abstract Zip Data}, 
    pdfauthor={Christopher Lang},       
}

\maketitle
\section{Introduction}

An important tool in the analysis of Shimura varieties is the theory of \(G\)-zips. In \cite{pink2011algebraiczipdata} and \cite{Pink_2015}, a method was developed to compute the underlying topological space of the stack of \(G\)-zips. For displays we have a very similar formalism, so we want to generalize this method to a more abstract setting.

The idea in the construction is that the stack of \(G\)-zips of type \(\mu\) has a description as a quotient stack \([E\backslash G]\), which is complicated to understand, so we do a refinement process to split the problems into smaller, easier problems. To obtain this refinement process, we look at the following datum: The cocharacter \(\mu\) fixes two parabolic subgroups \(P\) and \(Q\), and we have 
\[E=\set{(x,y)\in P\times Q}{\phi(\bar x)=\bar y},\]
where \(\bar x\) is the projection in the Levi quotient. From \(E\) we have the two projections \(\pi_1\) and \(\pi_2\) to \(G\), and we call \((E,G,\pi_1,\pi_2)\) an abstract zip datum. The action of \(E\) on \(G\) in the quotient stack \([E\backslash G]\) is given by \(e.g=\pi_1(e)g\pi_2(e)\inv\). Given such a zip datum, one can refine it by replacing \(E\) and \(G\) by certain smaller groups, and restricting the two morphisms.
Then one shows that in order to understand the orbits of \([E\backslash G]\), it is enough to understand the orbits of the refinement plus some double quotient. This double quotient will look something like \(P\backslash G/Q\) for parabolics \(P\) and \(Q\), and this is understood by the Bruhat decomposition, as the subset \({}^IW^J\) of the Weyl group of \(G\). Now the idea is to do this refinement process over and over, until the process becomes stationary at \([E_n\backslash G_n]\) where we only have one orbit, and then we piece together the orbits of \([E\backslash G]\) from the double quotients. 

We want to generalize such a construction to arbitrary groups \(E\) and \(G\), as well as two group homomorphisms \(\tau,\sigma: E\to G\) between them, and an action
\(E\times G\to G,\; (e,g)\mapsto \tau(e)g\sigma(e)\inv\). 
\begin{introdefinition}[Definition \ref{def_zip_data}]
    A \emph{zip datum} is a quadruple \(\Zcal = (E,G,\tau,\sigma)\) consisting of two (abstract) groups \(E,G\) and two group homomorphisms \(\tau,\sigma:E\to G\). Given a zip datum, we can look at the conjugated and the refined zip datum
    \begin{alignat*}{3}
        \Zcal^x &= (E,G,\tau,{}^x\sigma),&& x\in G,\\
        \Zcal_1 &= (E_1, G_1,\tau,\sigma) \coloneqq (\sigma\inv(\tau(E)),\tau(E),\tau,\sigma).
    \end{alignat*}
\end{introdefinition}
Now, instead of looking at the equivalence relation given by the action above, we look at the following coarser equivalence relation. To define it, we need the group \(G_\infty^x\coloneqq\tau(E_\infty^x)\), where \(E_\infty^x\) is the larges subgroup satisfying \({}^x\sigma(E_\infty^x)\subseteq\tau(E_\infty^x)\). In many cases, \(G_\infty^x\) is the intersection of the \(G\)'s from the repeated refinements of the zip datum \(\Zcal^x\).
\begin{introdefinition}[Definition \ref{def_equiv_rel_zips}]
    Given a zip datum \(\Zcal = (E,G,\tau,\sigma)\), we define an equivalence relation on \(G\) by saying \(y\sim x\) if there are \(e\in E\) and \(g_x\in G_\infty^x\) with \(y=\tau(e)g_x x\sigma(e)\inv\).
\end{introdefinition}
One can show that in some cases, for example for \(G\)-zips, the two equivalence relations have the same orbits, but in general this equivalence relation is strictly coarser. But with this equivalence relation we can show the following 
\begin{introproposition}[Proposition \ref{prop_zip_refinement_one_step}]
    Given a zip datum \(\Zcal = (E,G,\tau,\sigma)\) and \(x\in G\). The map 
    \[G_1^x \to \tau(E)x\sigma(E),\quad y\mapsto yx\]
    maps \(\mathcal Z_1^x\)-equivalence classes in \(G_1^x\) bijectively to \(\mathcal Z\)-equivalence classes in \(\tau(E)x\sigma(E)\).
\end{introproposition}
Using this proposition one can again see that in order to understand the entire equivalence relation, one needs to understand it on the smaller zip datum \(\Zcal_1^x\) and the double quotient \(\tau(E)\backslash G/\sigma(E)\). If the refinement process becomes stationary, then one can use this principle to show that the equivalence classes of \(\Zcal\) can be fully understood with the double quotients of the repeatedly refined zip data. This can be made precise in the following way: A system of representatives of the double quotients of the refined and twisted zip data can be made into a rooted forest, where a node in the \(i\)-th generation is a representative of a double quotient at the refinement step \(i\), see definition \ref{def_forest_double_quotient}. If we denote the \(i\)-th generation of the forest by \(R_i\) and by \(R_{i+1}\to R_i\) the map that sends a node to its parent, then we get the description
\[G/\sim_\Zcal \,\cong \lim\big( \cdots \to R_{i+1}\to R_i\to \cdots \to R_0\big),\]
as long as \(R_{i+1}\to R_i\) is an isomorphism for \(i\gg0\), see theorem \ref{thm_forest_double_quotient}.

\section*{Acknowledgements}
I want to thank my PhD advisor Torsten Wedhorn, who has mentored me throughout my journey of learning algebraic geometry over the last four years. You have been a great supervisor and always gave me a nice intuition.

This project was funded by the Deutsche Forschungsgemeinschaft (DFG, German Research Foundation) TRR 326 Geometry and Arithmetic of Uniformized Structures, project number 444845124.
\newpage
\section{Abstract Zip Data}
\label{chapter_zip_data}

The stack of \(G\)-zips can be understood as a quotient stack \([E\backslash G]\), where the group \(E\) acts on \(G\) by left and right multiplication with two different projections, namely \(E\) consists of pairs \((p,q)\) coming from parabolic subgroups, and this acts on \(G\) by \((p,q).g=pgq\inv\), see \cite[3.6]{Pink_2015}. In \cite{pink2011algebraiczipdata}, the authors constructed a refinement process for this quotient stacks, which yields an explicit description of the underlying topological space, showing that it equals \({}^IW\). We now want to generalize this refinement process to the setting of two arbitrary groups \(E\) and \(G\), as well as two group homomorphisms \(\tau,\sigma: E\to G\) between them, and an action
\(E\times G\to G,\; (e,g)\mapsto \tau(e)g\sigma(e)\inv\).

\begin{definition}
\label{def_zip_data}
    A \emph{zip datum} is a quadruple \(\Zcal = (E,G,\tau,\sigma)\) consisting of two (abstract) groups \(E,G\) and two group homomorphisms \(\tau,\sigma:E\to G\). Given a zip datum, we can look at the conjugated and the refined zip datum
    \begin{alignat*}{3}
        \Zcal^x &= (E,G,\tau,{}^x\sigma),&& x\in G,\\
        \Zcal_1 &= (E_1, G_1,\tau,\sigma) \coloneqq (\sigma\inv(\tau(E)),\tau(E),\tau,\sigma),
        \intertext{where \({}^x\sigma(e)\coloneqq x\sigma(e)x\inv\). Furthermore, we set}
        \Zcal_{i} &= (E_i,G_i,\tau,\sigma) \coloneqq (\Zcal_{i-1})_1,&& i>1,\\
        \Zcal_i^x &= (E_i^x,G_i^x,\tau,{}^x\sigma) \coloneqq (\Zcal^x)_i,&& x\in G.
    \end{alignat*}
\end{definition}

\begin{remark}
    Note that \(\tau\) and \(\sigma\) restrict to maps \(E_i\to G_i\), and these are the maps meant in the zip datum \(\Zcal_i\). When computing \(E_i = \sigma\inv(G_{i-1})\), it does not matter whether we view \(\sigma\) as a map defined on \(E\) or on \(E_{i-1}\). For example, for \(E_2\) one sees this due to the equation
    \[E_2 = \sigma\inv(\tau(E_1))\subseteq \sigma\inv(\tau(E)) = E_1.\]
    In particular, the two sequences \(E_i\) and \(G_i\) are decreasing.
    Also note that 
    \begin{align*}
        (\Zcal^x)_1 &= \bigl(\sigma\inv({}^{x\inv}\tau(E)),\tau(E),\tau,{}^x\sigma\bigr)
        \shortintertext{is not the same as}
        (\Zcal_1)^x &= \bigl(\sigma\inv(\tau(E)),\tau(E),\tau,{}^x\sigma\bigr),
    \end{align*}
    as the first component differs by conjugation under \(\sigma(x)\inv\).
\end{remark}

\begin{definition}
    Given a zip datum \(\Zcal = (E,G,\tau,\sigma)\), we denote by \(E_\infty = E_\infty(\Zcal)\) the largest subgroup of \(E\) with \(\sigma(E_\infty)\subseteq\tau(E_\infty)\).
    We write 
    \begin{align*}
        G_\infty = G_\infty(\Zcal) &\coloneqq \tau(E_\infty),\\
        E_\infty^x &\coloneqq E_\infty(\Zcal^x),\\
        G_\infty^x &\coloneqq G_\infty(\Zcal^x),\quad x\in G.
    \end{align*}
\end{definition}

\begin{proposition}
\label{E_inf_invariant_refinement}
    The groups \(E_\infty\) and \(G_\infty\) are invariant under refinement, i.e.\ for a zip datum \(\Zcal = (E,G,\tau,\sigma)\) we have 
    \[E_\infty(\Zcal) = E_\infty(\Zcal_1)\quad\text{and}\quad G_\infty(\Zcal) = G_\infty(\Zcal_1).\]
    In particular, we have \(E_\infty\subseteq E_i\) for all \(i\ge0\).
\end{proposition}
\begin{proof}
    It is enough to show that if we have a subgroup \(F\subseteq E\) with \(\sigma(F)\subseteq \tau(F)\), then we have \(F\subseteq E_1=\sigma\inv(\tau(E))\). But this follows from
    \(\sigma(F) \subseteq\tau(F) \subseteq\tau(E)\).
\end{proof}

\begin{example}
\label{display_GL2(W(R))}
    For a perfect field \(k\) in characteristic \(p\), take \(G=\GL_2(W(k))\) and \(E=\left\{\begin{pmatrix}
        * & * \\ p & *
    \end{pmatrix}\right\}\) with
    \begin{align*}
        \tau(e) &= e\\
        \sigma\begin{pmatrix}
            a & b \\ pc & d
        \end{pmatrix} &= \begin{pmatrix}
            p \\ & 1
        \end{pmatrix}\begin{pmatrix}
            \phi(a) & \phi(b) \\ p\phi(c) & \phi(d)
        \end{pmatrix}\begin{pmatrix}
            p\inv \\ & 1
        \end{pmatrix} = \begin{pmatrix}
            \phi(a) & p\phi(b) \\ \phi(c) & \phi(d)
        \end{pmatrix}
    \end{align*}
    Then \(G_i=\begin{pmatrix}
        * & * \\ p^i & *
    \end{pmatrix}\) and \(E_i=\begin{pmatrix}
        * & * \\ p^{i+1} & *
    \end{pmatrix}\) and \(G_\infty=E_\infty=\begin{pmatrix}
        * & * \\ 0 & *
    \end{pmatrix}\)

    Now let \(x=\begin{pmatrix}
        0 & 1 \\ 1 & 0
    \end{pmatrix}\). Then we have 
    \[{}^x\sigma\begin{pmatrix}
        a & b \\ pc & d
    \end{pmatrix} = \begin{pmatrix}
        \phi(d) & \phi(c) \\ p\phi(b) & \phi(a)
    \end{pmatrix}\]
    and hence \(G_i^x=E_i^x=\begin{pmatrix}
        * & * \\ p & *
    \end{pmatrix}\) for \(1\le i\le\infty\).

    In particular, we see that in these cases the group \(E_\infty\) is the intersection of the \(E_i\). But note that this does not need to be the case, as the following example shows.
\end{example}

\begin{example}
    Let us first work in the category of sets. Let \(E=[0,1]\), \(G=\RR_{\ge0}\) and
    \[\begin{tikzcd}
        {[0,1]} \rar[shift left, "\tau\,:\, x\,\mapsto\, (1+x)|\sin(1/x)|"] \rar[shift right, "\sigma\,:\, x\,\mapsto\, 1+2x"'] &[70pt] \RR_{\ge0}
    \end{tikzcd}\]
    and \(\tau(0)\coloneqq 0\). Then we have 
    \[G_i = [0,1+a_i],\qquad E_i=[0,b_i]\]
    with \(a_i,b_i>0\) both monotonously decreasing to \(0\), because \(a_i\le b_{i-1}\) and \(b_i = \frac{a_i}{2}\). But we have  \(E_\infty = \emptyset\).

    One can upgrade this example to the category of groups by looking at the free abelian group generated by \(E\) resp.\ \(G\), i.e.\ we set \(\hat E = \ZZ^{(E)}\) and \(\hat G = \ZZ^{(G)}\) with \(\hat\tau,\hat\sigma:\hat E\to \hat G\) such that \(\hat\tau|_E=\tau\) and \(\hat\sigma|_E=\sigma\). Then we have 
    \[\hat G_i = \ZZ^{(G_i)}\qquad\text{and}\qquad \hat E_i = \ZZ^{(E_i)}.\]
    This follows because 
    \[\hat\tau(\hat E_i) = \hat\tau(\ZZ^{(E_i)}) = \ZZ^{(\tau(E_i))}\qquad \text{and}\qquad \hat\sigma\inv(\hat G_i) = \hat\sigma\inv(\ZZ^{(G_i)}) = \ZZ^{(\sigma\inv(G_i))},\]
    where the last equality holds as \(\sigma\) is injective. In particular, we have \(\bigcap_{i\ge0}\hat E_i = \ZZ^{(\{0\})}\). But one quickly checks \(\ZZ^{(\{1\})}=\hat\sigma(\ZZ^{(\{0\})})\not\subseteq \hat\tau(\ZZ^{(\{0\})})=\ZZ^{(\{0\})}\), hence \(\hat E_\infty = \{0\}\).
\end{example}

The reason for this failure is that in general we have \(\tau(\bigcap_i E_i) \not= \bigcap_i \tau(E_i)\).

\begin{proposition}
    Given a zip datum \(\Zcal = (E,G,\tau,\sigma)\). If we have \(\tau(\bigcap_i E_i) = \bigcap_i \tau(E_i)\), then we get 
    \[E_\infty = \bigcap_{i\ge0} E_i \qquad\text{and}\qquad G_\infty = \bigcap_{i\ge0} G_i\]
\end{proposition}
\begin{proof}
    By proposition \ref{E_inf_invariant_refinement} we already know "\(\subseteq\)" in the first equation. We need to show that \(\sigma(\bigcap_i E_i)\subseteq \tau(\bigcap_i E_i)\). But we have 
    \[\sigma\left(\bigcap\nolimits_i E_i\right) \subseteq \sigma(E_i) = \sigma(\sigma\inv(\tau(E_{i-1}))) \subseteq \tau(E_{i-1})\]
    and therefore \(\sigma(\bigcap_i E_i) \subseteq \bigcap_i \tau(E_i) = \tau(\bigcap_i E_i)\).
    For the second equality we have
    \[G_\infty = \tau(E_\infty) = \tau\left(\bigcap\nolimits_i E_i\right) = \bigcap\nolimits_i \tau(E_i) = \bigcap\nolimits_i G_{i+1}.\]
\end{proof}

\begin{remark}
    This condition is in particular satisfied, if one (equivalently: both) of the sequences \(E_i\) or \(G_i\) becomes stationary. This happens for example if one of these groups is finite, or if we look at group schemes instead of abstract groups and the \(E_i\) or \(G_i\) are closed subgroups inside a Noetherian scheme. This is the reason why the refinement process becomes stationary in \cite{pink2011algebraiczipdata}. But note that this does not hold anymore if we look at displays instead of zips, as example \ref{display_GL2(W(R))} shows. However, in this example, \(\tau\) is injective, so the condition \(\tau(\bigcap_i E_i) = \bigcap_i \tau(E_i)\) is again satisfied.
\end{remark}

\begin{lemma}
\label{E_infty_multiplicative}
    Given a zip datum \(\Zcal = (E,G,\tau,\sigma)\) and \(x,y\in G\). 
    \begin{enumerate}
        \item If \(y\in G_1=\tau(E)\), then \(\Zcal_1^{yx} = (\Zcal_1^x)^y\) and in particular 
        \[E_\infty^{yx}(\mathcal Z)=E_\infty^y(\mathcal Z_1^x)\qquad\text{and}\qquad G_\infty^{yx}(\mathcal Z)=G_\infty^y(\mathcal Z_1^x).\]
        \item If \(x=y\) in \(\tau(E)\backslash G / \sigma(E)\), i.e.\ \(y=\tau(e)x\sigma(\tilde e)\), then \(E_1^y = {}^{\tilde e\inv}E_1^x\).
    \end{enumerate}
\end{lemma}
\begin{proof}
    \begin{enumerate}
        \item The only thing we need to verify is that \(E_1^x = E_1^{yx}\), as then both zip data are just the zip datum \((E_1^x, \tau(E), \tau, {}^{yx}\sigma)\). But this follows from
        \begin{align*}
            E_1^{yx}(\mathcal Z)
        = ({}^{yx}\sigma)\inv(\tau(E))
        = \sigma\inv({}^{x\inv y\inv}\tau(E))
        = \sigma\inv({}^{x\inv \tau(e_y)\inv}\tau(E))
        &= \sigma\inv({}^{x\inv}\tau({}^{e_y}E))\\
        &= ({}^x\sigma)\inv(\tau(E))\\
        &= E_1^x,
        \end{align*}
        where \(e_y\in E\) with \(\tau(e_y)=y\).
        \item We have 
    \begin{align*}
        E_1^y &= \sigma\inv\big({}^{y\inv}\tau(E)\big)\\
        &= \sigma\inv\big({}^{\sigma(\tilde e)\inv x\inv\tau(e)\inv}\tau(E)\big)\\
        &= \sigma\inv\big({}^{\sigma(\tilde e)\inv x\inv}\tau(E)\big)\\
        &= {}^{\tilde e\inv}\sigma\inv\big({}^{x\inv}\tau(E)\big)\\
        &= {}^{\tilde e\inv}E_1^x. \qedhere
    \end{align*}
    \end{enumerate}
\end{proof}

\begin{definition}
\label{def_equiv_rel_zips}
    Given a zip datum \(\Zcal = (E,G,\tau,\sigma)\), we define an equivalence relation on \(G\) by saying \(y\sim x\) if there are \(e\in E\) and \(g_x\in G_\infty^x\) with \(y=\tau(e)g_x x\sigma(e)\inv\). If the zip datum may be ambiguous, we also write \(y \sim_\Zcal x\). We write \(o_\Zcal(x)\subseteq G\) for the equivalence class of \(x\in G\).
\end{definition}
\begin{proof}
    We show that this indeed defines an equivalence relation.
    Reflexivity is clear. For symmetry assume \(y=\tau(e)g_x x\sigma(e)\inv\). Choose \(e_x\in\tau\inv(g_x)\subseteq E_\infty^x\) and observe
    \[{}^y\sigma({}^eE_\infty^x)={}^{y\sigma(e)}\sigma(E_\infty^x)={}^{\tau(e)g_x x}\sigma(E_\infty^x) \subseteq {}^{\tau(e)g_x}\tau(E_\infty^x) = \tau({}^{e e_x}E_\infty^x) = \tau({}^e E_\infty^x)\]
    and hence \({}^eE_\infty^x \subseteq E_\infty^y\). Define \(e_y\coloneqq {}^e e_x\inv\in E_\infty^y\). With \(g_y=\tau(e_y)\) and \(\tilde e=e\inv\) we have 
    \[x = \tau(e_x)\inv\tau(e)\inv y \sigma(e) = \tau({}^{e\inv}e_y) \tau(e)\inv y \sigma(e) = \tau(e)\inv \tau(e_y) y \sigma(e) = \tau(\tilde e) g_y y \sigma(\tilde e)\inv,\]
    hence \(x\sim y\). For transitivity assume \(y=\tau(e_1)g_x x\sigma(e_1)\inv\) and \(z=\tau(e_2)g_y y\sigma(e_2)\inv\). Set \(e=e_2e_1\), then  
    \[z = \tau(e_2)g_y \tau(e_1)g_x x\sigma(e_1)\inv\sigma(e_2)\inv = \tau(e) \tau({}^{e_1\inv}e_y) g_x x \sigma(e)\inv\]
    and \({}^{e_1\inv}e_y\in {}^{e_1\inv}E_\infty^y\subseteq E_\infty^x\) by the same reason as before, so \(z\sim x\).
\end{proof}

\begin{remark}
    The proof shows that if \(y\sim x\) with \(y=\tau(e)g_x x\sigma(e)\inv\), then we have \({}^e E_\infty^x = E_\infty^y\).
\end{remark}

\begin{remark}
    Given a zip datum \(\Zcal = (E,G,\tau,\sigma)\), we can look at the quotient groupoid \([E\backslash G]\) given by the action \(e.g\coloneqq \tau(e)g\sigma(e)\inv\). The equivalence relation \(\sim_\Zcal\) is coarser than the one given by this group action, so we get a map \(|[E\backslash G]|\to G/\sim_\Zcal\). In particular, if we know the equivalence classes of \(\sim_\Zcal\), we get a decomposition of \(G\) by unions of orbits of the group action. In the case of \(G\)-zips, each equivalence class is even a single orbit, so the map above is bijective, see \cite[9.18]{pink2011algebraiczipdata}.
\end{remark}

\begin{example}
    If \(\tau\) is surjective, there is only one equivalence class, as in this case \(G=G_\infty^x\) for all \(x\in G\) and we can choose \(e\) such that \(y=\tau(e)\) and then \(g_x\coloneqq \sigma(e)x\inv\). This is in particular the case for the zip datum \(\Zcal_\infty^x\coloneqq(E_\infty^x,G_\infty^x,\tau,{}^x\sigma)\).
\end{example}

\begin{proposition}
\label{prop_zip_refinement_one_step}
    Given a zip datum \(\Zcal = (E,G,\tau,\sigma)\) and \(x\in G\). The map 
    \[\Psi_\Zcal^x: \underbrace{G_1^x}_{=\tau(E)} \to \tau(E)x\sigma(E),\quad y\mapsto yx\]
    maps \(\mathcal Z_1^x\)-equivalence classes in \(G_1^x\) bijectively to \(\mathcal Z\)-equivalence classes in \(\tau(E)x\sigma(E)\). Furthermore, we have 
    \[o_{\mathcal Z_1^x}(y)\cdot x = o_{\mathcal Z}(yx) \cap G_1^x\cdot x ,\qquad y\in G_1^x.\]
\end{proposition}
\begin{proof}
    We first show the equation. An element \(z\in G_1^x =\tau(E)\) is in \(o_{\mathcal Z_1^x}(y)\) if and only if there are \(e\in E_1^x\) and \(g_y\in G_\infty^y(\mathcal Z_1^x) \overset{\ref{E_infty_multiplicative}}{=} G_\infty^{yx}\) such that \(z = \tau(e)g_y y ({}^x\sigma(e))\inv\). On the other hand, \(zx\) is in \(o_{\mathcal Z}(yx)\) if and only if there are \(\tilde e\in E\) and \(g_{yx}\in G_\infty^{yx}\) with \(zx = \tau(\tilde e)g_{yx}yx\sigma(\tilde e)\inv\), or equivalently \(z = \tau(\tilde e)g_{yx}y({}^x\sigma(\tilde e))\inv\). Hence we want to choose \(e=\tilde e\) and \(g_y=g_{yx}\), and the only thing we need to verify for this is that \(\tilde e\in E_1^x=({}^x\sigma)\inv(\tau(E))\). But after rearranging we have \({}^x\sigma(\tilde e) = z\inv\tau(\tilde e)g_{yx}y\in\tau(E)\), as every factor is in \(\tau(E)\), so we are done. 

    Now we show the first claim. The equation we just showed implies that the map on equivalence classes is well-defined and injective. 
    But an element of the form \(\tau(e)x\sigma(\tilde e)\) is \(\mathcal Z\)-equivalent to 
    \[\tau(\tilde e)\cdot1\cdot\tau(e)x\sigma(\tilde e)\sigma(\tilde e)\inv = \tau(\tilde e e)x,\]
    which lies in the image of the map, so we are also surjective on equivalence classes.
\end{proof}

This proposition suggests that if we change \(x\) inside one double coset, then the refined zip datum should not change too much. We want to make this precise next.

\begin{lemma}
\label{double_quotient_conjugation}
    If \(G\) is a group, \(H,K\subseteq G\) are subgroups and \(x,y\in G\), then 
    \begin{align*}
        H\backslash G/K &\xrightarrow{\sim} {}^xH\backslash G/{}^yK,\\
        g&\mapsto xgy\inv
    \end{align*}
    is a well-defined bijection.
\end{lemma}
\begin{proof}
    We need to check that \(g\) and \(hgk\) are mapped to equivalent elements for \(h\in H\) and \(k\in K\). This follows from
    \[xhgky\inv = xhx\inv xgy\inv yky\inv = ({}^xh) xgy\inv ({}^yk)\sim xgy\inv.\]
    The inverse is just \(g\mapsto x\inv gy\).
\end{proof}

\begin{lemma}
\label{functor_quotient_groupoid}
    Given groups \(G,H\) and a \(G\)-set \(X\), an \(H\)-set \(Y\), a group homomorphism \(\psi:G\to H\) and a map \(f:X\to Y\), then we obtain a functor of groupoids \([G\backslash X]\to [H\backslash Y]\) if \(f(g.x) = \psi(g).f(x)\) holds for all \(g\in G\) and \(x\in X\).
\end{lemma}
\begin{proof}
    On objects, the functor maps \(x\) to \(f(x)\), so we need to check that if we have 
    \[g\in \Hom(x,x') = \set{g\in G}{g.x=x'},\]
    then we get 
    \[\psi(g)\in \Hom(f(x),f(x')),\]
    and for this one needs 
    \(\psi(g).f(x)\overset{!}{=}f(x')=f(g.x)\). Morphisms are compatible with composition as \(\psi\) is a group homomorphism, so we indeed get a functor.
\end{proof}

\begin{proposition}
    If \(x=y\) in \(\tau(E)\backslash G/\sigma(E)\), with \(y=\tau(e)x\sigma(\tilde e)\), then
    \begin{align*}
        \Psi: (E_1^x, G_1^x, \tau, {}^x\sigma) &\to (E_1^y, G_1^y, \tau, {}^y\sigma)\\
        E_1^x\ni \varepsilon &\mapsto {}^{\tilde e\inv}\varepsilon \in E_1^y\\
        G_1^x\ni g &\mapsto \tau(\tilde e)\inv g \underbrace{x \sigma(\tilde e)y\inv}_{=\tau(e)\inv} = \tau(\tilde e)\inv g xy\inv \cdot{}^y\sigma(\tilde e) \in G_1^y
        \intertext{induces the bijection}
        \tau(E_1^x)\backslash G_1^x /{}^x\sigma(E_1^x) &\xrightarrow{\sim} \tau(E_1^y)\backslash G_1^y /{}^y\sigma(E_1^y)
    \end{align*}
    from lemma \ref{double_quotient_conjugation}, and \(\Psi\) induces an equivalence of groupoids \([E_1^x\backslash G_1^x]\xrightarrow{\sim} [E_1^y\backslash G_1^y]\).
\end{proposition}
\begin{proof}
    For the bijection we need to check that 
    \[{}^{\tau(\tilde e)\inv}\tau(E_1^x) = \tau(E_1^y) 
    \qquad\text{and}\qquad
    {}^{\tau(e)}\bigl( {}^x\sigma(E_1^x) \bigr) = {}^y\sigma(E_1^y).\]
    The first one follows directly from the equation \(E_1^y = {}^{\tilde e\inv} E_1^x\) from lemma \ref{E_infty_multiplicative}. For the second equation we compute
    \[{}^{\tau(e)}\bigl( {}^x\sigma(E_1^x) \bigr) = {}^{\tau(e)x}\sigma({}^{\tilde e}E_1^y) = {}^{y\sigma(\tilde e)\inv}\sigma({}^{\tilde e}E_1^y) = {}^y\sigma(E_1^y).\]
    By lemma \ref{functor_quotient_groupoid}, such a map \(\Psi\) induces a map of quotient groupoids if we have \(\Psi(\varepsilon.g) = \Psi(\varepsilon).\Psi(g)\). This equality holds as
    \begin{align*}
        \Psi(\varepsilon).\Psi(g) 
        &= \tau({}^{\tilde e\inv}\varepsilon) \tau(\tilde e)\inv g\tau(e)\inv \cdot{}^y\sigma({}^{\tilde e\inv} \varepsilon)\inv\\
        &= \tau(\tilde e)\inv \tau(\varepsilon) g x\sigma(\tilde e) \sigma({}^{\tilde e\inv} \varepsilon)\inv y\inv\\
        &= \tau(\tilde e)\inv \tau(\varepsilon) g x\sigma(\varepsilon)\inv \sigma(\tilde e) y\inv
        \shortintertext{and}
        \Psi(\varepsilon.g) 
        &= \tau(\tilde e)\inv \tau(\varepsilon) g \cdot{}^x\sigma(\varepsilon)\inv \tau(e)\inv\\
        &= \tau(\tilde e)\inv \tau(\varepsilon) g x \sigma(\varepsilon)\inv x\inv \tau(e)\inv\\
        &= \tau(\tilde e)\inv \tau(\varepsilon) g x \sigma(\varepsilon)\inv \sigma(\tilde e) y\inv. \qedhere
    \end{align*}
\end{proof}

Now we can describe the \(\Zcal\)-equivalence classes using the refinement process. This description will be helpful, if we understand the double quotients \(\tau(E)\backslash G/\sigma(E)\). If \(G\) is a matrix group, these can often be understood with Bruhat decompositions, and then these double quotients will be subsets of the Weyl group of \(G\). The idea now is that if we want to understand the equivalence classes, we need to understand the double quotient, and each point in the double quotient again looks like a zip datum, which we investigate in the same way. So in the end, if we want to specify an equivalence class, we can equivalently give a sequence of elements of certain double quotients. This is the same idea as if you want to specify where you are in the world, you give the continent, then the country, then the state and so on. 

In order to make this description formal, we use the language of a \emph{rooted forest} from graph theory, which is a graph without cycles, together with one distinguished node for each connected component. These nodes are called the roots of the forest. We say that a node lives in generation \(n\), if the unique path to the root of its connected component has length \(n\). In particular, the roots live in generation 0. Every rooted forest \(R\) has a unique decomposition \(R=\coprod_{n\in\NN}R_n\) into its generations.
\begin{definition}
\label{def_forest_double_quotient}
    A \emph{system of representatives} \(R=\coprod_{n\in\NN} R_n\) for a zip datum \(\Zcal = (E,G,\tau,\sigma)\) is a rooted forest with vertices in \(G\), such that the roots are a system of representatives of \(\tau(E)\backslash G/\sigma(E)\) and a vertex \(x_n\in R_n\) has as children a system of representatives of \(\tau(E_{n+1}^{\tilde x_n})\backslash G_{n+1}^{\tilde x_n}/{}^{\tilde x_n}\sigma(E_{n+1}^{\tilde x_n})\), where \(\tilde x_n = x_n\cdots x_0\) is the product along the unique path to the root of the tree (\(x_{i}\) is the parent of \(x_{i+1}\)).
\end{definition}

\begin{theorem}
\label{thm_forest_double_quotient}
    Given a zip datum \(\Zcal = (E,G,\tau,\sigma)\) with a system of representatives \(R\). Assume that \(R\) becomes stationary in the sense that \(R_N=R_{N+1}\) for \(N\gg0\). Then we have a bijection
    \[G/{\sim_\Zcal} \cong \set{(r_n)\in \prod_{n\in\NN}R_n}{r_n  \text{ is the parent of }r_{n+1}} = \lim_n R_n,\]
    where the transition maps are given by mapping an element to its parent.
\end{theorem}
\begin{proof}
    The bijection is given by mapping an \(x\in G\) to the vector \((r_n)\) characterized in the following way: The zeroth entry is the unique \(r_0\in R_0\) such that \(x=r_0\) in \(\tau(E)\backslash G/\sigma(E)\). The children of \(r_0\) are a system of representatives for the double quotient of the zip datum \(\Zcal_1^{r_0}\), and \(r_1\) is the representative of \(x r_0\inv\) in this double quotient, compare proposition \ref{prop_zip_refinement_one_step}. In general, \(r_n\) is the representative of \(x r_0\inv \cdots r_{n-1}\inv\) in the double quotient of \(\Zcal_n^{r_{n-1}\cdots r_0}\). The inverse is given by \((r_n)\mapsto r_N\cdots r_0\) for \(N\gg0\).
\end{proof}

\begin{example}
    Let's apply this theorem to the zip datum of example \ref{display_GL2(W(R))}, namely \(G=\GL_2(W(k))\) and \(E=\left\{\begin{psmallmatrix}
        * & * \\ p & *
    \end{psmallmatrix}\right\}\). To compute \(R_0\) we must understand 
    \[\tau(E)\backslash G/ \sigma(E) = 
    \left\{\begin{psmallmatrix}
        * & * \\ p & *
    \end{psmallmatrix}\right\}
    \backslash \GL_2(W(k)) /
    \left\{\begin{psmallmatrix}
        * & p \\ * & *
    \end{psmallmatrix}\right\}.\]
    This is a special case of the (affine) Bruhat decomposition from \cite[5.1]{kaletha2023bruhat}, namely it has two elements, represented by the identity matrix and \(x=\begin{psmallmatrix}
        0 & 1 \\ 1 & 0
    \end{psmallmatrix}\). For these two points we already computed the refined zip data. For \(x\) we saw that \(\tau\) is surjective on \(\Zcal_1^x\), so the double quotient is trivial (as well as the double quotients for all the refined ones). For the identity matrix the situation is less clear, as the refinements of the zip datum don't become stationary. The double quotient of the \(i\)th refinement is 
    \[\left\{\begin{pmatrix}
        * & * \\ p^{i+1} & * 
    \end{pmatrix}\right\} \scalebox{2}{\(\backslash\)}
    \raisebox{5pt}{\(
    \left\{\begin{pmatrix}
        * & * \\ p^i & *
    \end{pmatrix}\right\}
    \)}
    \scalebox{2}{\(/\)} 
    \left\{\begin{pmatrix}
        * & p \\ p^i & *
    \end{pmatrix}\right\},\]
    and we claim that all of them only consist of one element. For this it is enough to show that given a matrix \(\begin{psmallmatrix}
        a & b \\ p^i c & d
    \end{psmallmatrix}\), we can find a matrix \(\begin{psmallmatrix}
        \alpha & \beta \\ p^{i+1}\gamma & \delta
    \end{psmallmatrix}\), such that 
    \[\begin{pmatrix}
        \alpha & \beta \\ p^{i+1}\gamma & \delta
    \end{pmatrix} \cdot
    \begin{pmatrix}
        a & b \\ p^i c & d
    \end{pmatrix} \in \left\{\begin{pmatrix}
        * & p \\ p^i & *
    \end{pmatrix}\right\}.\]
    All we need to do for this is to show that the upper left entry, which equals \(\alpha b + \beta d\), is divisible by \(p\). Note that \(\alpha\) must be a unit, but \(\beta\) can be chosen arbitrary, so we can set \(\alpha=\delta=1\), \(\gamma=0\) and \(\beta=-bd\inv\). Hence all these double quotients are trivial and we see that \(R_n=R_0\) for all \(n>0\), so that there are exactly two \(\sim_\Zcal\)-equivalence classes.
\end{example}

\begin{example}
    The "induction step" from \cite[4]{pink2011algebraiczipdata}, which is used to describe the topological space of \(G\Zip^\mu\), see \cite[5.10]{pink2011algebraiczipdata}, almost fits into our framework. We have the zip datum
    \begin{align*}
        E_\mu=\set{(p,q)\in P\times Q}{\phi(\bar p) = \bar q} &\to G\\
        \tau: (p,q) &\mapsto p\\
        \sigma: (p,q) &\mapsto q.
    \end{align*}
    In our language, the group \(G_1\) would be \(P\), but in the "induction step" of \cite{pink2011algebraiczipdata} the authors additionally take the quotient by the unipotent radical (and do a conjugation), in order to obtain a zip datum where the group \(G\) is again reductive. This makes the description of the double quotients easier, as we can use the Bruhat decomposition in every refinement step. Let's make the refined zip datum from the "induction step" precise: 

    Let \(I,J\) be the types of \(P,Q\) and \(L,M\) be the Levi quotients. Let \(y\in {}^{\phi\inv J}W^I\) with \(J=\phi({}^yI)\), so that \((T,B,\dot y)\) is a frame (for \(P,Q\)) in the sense of \cite[3.6]{pink2011algebraiczipdata}. For \(x\in {}^IW^J\cong P\backslash G/Q\) we set 
    \[P_1^x \coloneqq M\cap {}^{\dot x\inv \dot y\inv}P \quad\text{and}\quad Q_1^x \coloneqq \phi(L\cap {}^{\dot y \dot x}Q)\]
    and consider the corresponding zip datum 
    \begin{align*}
        \tilde E_1^x=\set{(p,q)\in P_1^x\times Q_1^x}{\phi({}^{\dot y\dot x}\bar p) = \bar q} &\to M\eqqcolon \tilde G_1^x\\
        \tau: (p,q) &\mapsto p\\
        \sigma: (p,q) &\mapsto q.
    \end{align*}
    We have \(I_1^x=J\cap {}^{x\inv}I\) and \(J_1^x=\phi({}^yI\cap {}^{yx}J)\), the types of \(P_1^x\) and \(Q_1^x\), see \cite[4.13]{pink2011algebraiczipdata}.
    
    The description of the orbits of the equivalence relation through double quotients works just like in theorem \ref{thm_forest_double_quotient}. So the 0th generation is \({}^IW^J\cong P\backslash G/Q\), the children of the node \(x\in {}^IW^J\) are \({}^{I_1^x}W_J{}^{J_1^x}\) and so on. 
    
    In \cite{pink2011algebraiczipdata} the description of the topological space was obtained by an inductive argument, and not by looking at the entire forest of double quotients. The induction is taken over the number of steps until the refinement process becomes stationary, and in the induction step the combinatoric formula result was used, that \({}^IW\) is the set of all \(xw\) with \(x\in {}^IW^J\) and \(w\in {}^{I_x}W_J\).
\end{example}

We end the discussion of zip data with two results on the structure of the groups \(E_\infty\) and \(G_\infty\).
\begin{lemma}
\label{description_E_infty_via_G_infty}
    Given a zip datum \(\Zcal = (E,G,\tau,\sigma)\), then we have
    \[E_\infty = \set{e\in E}{\sigma(e)\in G_\infty\tau(e) G_\infty}.\]
\end{lemma}
\begin{proof}
    If \(e\in E_\infty\), then \(\sigma(e)\in \tau(E_\infty)=G_\infty\), and hence 
    \[\sigma(e) = \underbrace{\tau(e\inv)}_{\in G_\infty} \tau(e) \underbrace{\sigma(e)}_{\in G_\infty}.\]
    For "\(\supseteq\)" suppose that \(e\in E\) with \(\sigma(e)\in G_\infty\tau(e) G_\infty\). Now look at the group \(F\) generated by \(e\) and \(E_\infty\). We want to show that it satisfies \(\sigma(F)\subseteq\tau(F)\), so that \(e\in E_\infty\). For this it satisfies to show that \(\sigma(e)\in\tau(F)\), as we already have \(\sigma(E_\infty)\subseteq\tau(E_\infty)\) and all elements of \(F\) can be built by \(e\), \(e\inv\) and \(E_\infty\). But 
    \[\sigma(e) \in \underbrace{G_\infty}_{=\tau(E_\infty)\subseteq\tau(F)}\quad \underbrace{\tau(e)}_{\in\tau(F)}\quad \underbrace{G_\infty}_{=\tau(E_\infty)\subseteq\tau(F)}\subseteq \tau(F). \qedhere\]
\end{proof}

\begin{proposition}
    The group action 
    \begin{align*}
        E_\infty^x\times (E\times G_\infty^x) &\to E\times G_\infty^x \\
        \big(\varepsilon, (e,g)\big) &\mapsto \varepsilon.(e,g)\coloneqq\big(e\varepsilon\inv, \tau(\varepsilon)g\;{}^x\sigma(\varepsilon)\inv\big)
    \shortintertext{makes the map}
        E\times G_\infty^x &\twoheadrightarrow o_{\mathcal Z}(x)\\
        (e,g) &\mapsto \tau(e)gx\sigma(e)\inv
    \end{align*}
    into an \(E_\infty^x\)-torsor.
\end{proposition}
\begin{proof}
    We need to show that \((e,g)\) and \(\varepsilon.(e,g)\) have the same image in \(o_{\mathcal Z}(x)\) and that if \((e,g)\) and \((\tilde e, \tilde g)\) have the same image in \(o_{\mathcal Z}(x)\), then there is a unique \(\varepsilon\in E_\infty^x\) with
    \[(\tilde e, \tilde g) = \varepsilon.(e,g) = \big(e\varepsilon\inv, \tau(\varepsilon)g\;{}^x\sigma(\varepsilon)\inv\big).\]
    The first is just a quick computation, and for the second claim we are forced to define \(\varepsilon=\tilde e\inv e\). Then the equation we need to show is true because by assumption we have 
    \[\tau(\tilde e)\tilde gx\sigma(\tilde e)\inv = \tau(e)gx\sigma(e)\inv\]
    and hence 
    \[\tilde g = \tau(\underbrace{\tilde e\inv e}_{=\varepsilon}) gx\sigma(\underbrace{\tilde e\inv e}_{=\varepsilon})\inv x\inv.\]
    All that's left is to verify that we have \(\varepsilon\in E_\infty^x\). But for this we use that we have 
    \[E_\infty^x \overset{\ref{description_E_infty_via_G_infty}}{=} \set{e\in E}{{}^x\sigma(e)\in G_\infty^x\tau(e) G_\infty^x}\]
    after observing \({}^x\sigma(\varepsilon)=\tilde g\inv\tau(\varepsilon)g\). 
\end{proof}


\newpage
\printbibliography

@article {Pink_2015,
    AUTHOR = {Pink, Richard and Wedhorn, Torsten and Ziegler, Paul},
     TITLE = {{$F$}-zips with additional structure},
   JOURNAL = {Pacific J. Math.},
  FJOURNAL = {Pacific Journal of Mathematics},
    VOLUME = {274},
      YEAR = {2015},
    NUMBER = {1},
     PAGES = {183--236},
      ISSN = {0030-8730,1945-5844},
   MRCLASS = {14F05 (14F40 14G15 18D10)},
  MRNUMBER = {3347958},
MRREVIEWER = {Eva\ Viehmann},
       DOI = {10.2140/pjm.2015.274.183},
      eprint={1208.3547},
      archivePrefix={arXiv},
}

@article {pink2011algebraiczipdata,
    AUTHOR = {Pink, Richard and Wedhorn, Torsten and Ziegler, Paul},
     TITLE = {Algebraic zip data},
   JOURNAL = {Doc. Math.},
  FJOURNAL = {Documenta Mathematica},
    VOLUME = {16},
      YEAR = {2011},
     PAGES = {253--300},
      ISSN = {1431-0635,1431-0643},
   MRCLASS = {14L30 (14M27 20G15)},
  MRNUMBER = {2804513},
MRREVIEWER = {Eva\ Viehmann},
      eprint={1010.0811},
      archivePrefix={arXiv},
}

@book {kaletha2023bruhat,
    AUTHOR = {Kaletha, Tasho and Prasad, Gopal},
     TITLE = {Bruhat-{T}its theory---a new approach},
    SERIES = {New Mathematical Monographs},
    VOLUME = {44},
 PUBLISHER = {Cambridge University Press, Cambridge},
      YEAR = {2023},
     PAGES = {xxx+718},
      ISBN = {978-1-108-83196-3},
   MRCLASS = {20E42 (11F70 20G25 22E50)},
  MRNUMBER = {4520154},
MRREVIEWER = {Corina\ Ciobotaru},
       DOI = {10.1017/9781108933049},
}
\addcontentsline{toc}{section}{References}


\end{document}